\documentclass[a4paper,12pt]{amsart}
\usepackage{amsmath,amsthm}
\usepackage{amsfonts,amsmath,amsthm, amssymb, mathrsfs}
\usepackage{float}
\usepackage{euscript,eufrak,verbatim}
\usepackage{graphicx}
\usepackage[usenames]{color}
\usepackage[colorlinks,linkcolor=red,anchorcolor=blue,citecolor=blue]{hyperref}
\usepackage{amsmath}
\usepackage{amsthm}
\usepackage[all]{xy}
\usepackage{graphicx} 
\usepackage{amssymb} 
\usepackage{enumerate}

\newtheorem{thm}{Theorem}[section]
\newtheorem{prop}[thm]{Proposition}

\newtheorem{lem}[thm]{Lemma}

\def\N{\mathbb{N}}

\def\N{\mathbb{N}}

\def\diam{\text{\rm diam}}

\newtheorem*{thma}{Theorem A}

\def\logf{\log\frac{1}{\epsilon}}
\makeatletter 
\@addtoreset{equation}{section}
\numberwithin{equation}{section}

\title{Variational principles for Feldman-Katok metric mean dimension}

\author{Yunxiang Xie, Ercai Chen and Rui Yang*
}
\address
{School of Mathematical Sciences and Institute of Mathematics, Nanjing Normal University, Nanjing 210023, Jiangsu, P.R.China}
\email{yxxie20@126.com}
\email{ecchen@njnu.edu.cn}
\email{zkyangrui2015@163.com}
\date{}
\begin{document}
	
\maketitle

\renewcommand{\thefootnote}{}
\footnote{*corresponding author}

\begin{abstract}	
  We  introduce the notion of Feldman-Katok metric mean dimensions in this note.  We  show  metric mean dimensions defined by different metrics  coincide under weak tame growth of covering numbers,   and  establish  variational principles  for Feldman-Katok metric mean dimensions in terms of  FK Katok $\epsilon$-entropy and  FK  local $\epsilon$-entropy function. 
\end{abstract}

\section{Introduction}
  By a pair $(X,T)$  we mean a topological dynamical system (TDS for short),  where $X$ is a compact metric space with  metric $d$ and  $T $ is a homeomorphism on $X$. By $M(X), M(X,T), E(X,T)$ we denote the sets of Borel probability measures on $X$, $T$-invariant Borel probability measures on $X$, $T$-invariant  ergodic Borel probability measures on $X$, respectively.

  For  each  TDS, one can assign a non-negative number  to characterize the topological complexity of  system.  It  is well-known that  the classical topological entropy, defined by Bowen dynamical balls,   is an important topological invariant to help us understand the dynamical systems.  Besides, the Bowen dynamical  balls  with mistake function \cite{ps07},  dynamical balls defined by mean metric \cite{GJ16} and   Feldman-Katok metric \cite{CL21},   are also invoked and do not change the value of classical topological entropy.    It turns out that these different dynamical balls   become a   critical role in solving the problems toward Sarnak's  conjecture, multifractal analysis, the classification problems of measure-preserving systems and the other fields.  

   The present note  mainly involves  Feldman-Katok metric.   In \cite{O74,K43,F76}, the authors showed that edit distance  is  closely associated with the classification problems of measure-preserving systems.  Feldman-Katok metric  \cite{KL17}(FK metric for short)  is  the topological counterpart of edit distance.  Replacing Bowen metric by  FK metric,  Cai and Li \cite{CL21} proved that  topological entropy defined by FK metric coincides classical topological entropy \cite{w82}. Later,  Nie and Huang  \cite{NH22}  investigated the restricted sensitivity, return time and local  Brin-Katok entropy in context of FK metric.  Different from Bowen metric,   the work  \cite{KL17,KL17,NH22} suggests that the advantage to use  FK metric  is that it allows time delay by ignoring the synchronization of points in orbits with only order preserving required.  
   
   A compact metric space  $(X,d) $  is said to  have \emph{tame growth of covering numbers} if for each $\theta>0$,
   $$ \lim\limits_{\epsilon \to 0}\epsilon^{\theta}\log \#(X,d,\epsilon)=0,$$
   where $ \#(X,d,\epsilon)$ denotes the  smallest cardinality of  open balls $B_d(x,\epsilon)$ covering $X$. For example,  the compact subsets of $\mathbb{R}^n$  equipped with the Euclidean distance  have tame growth of covering numbers. More generally, it is shown that   \cite[Lemma 4]{lt18} every compact  metrizable topological space  admits a distance  satisfying such condition.   
   
   For TDS with infinite topological entropies, Lindenstrauss and Weiss \cite{lw00} introduced  the notion of metric mean dimension to  classify  such dynamical systems  and  established analogous variational principles for metric mean dimensions in terms of $L^{\infty}$-rate distortion functions and $L^{p} (1\leq p<\infty)$-rate distortion functions under the assumption of tame growth of covering numbers.
   \begin{thma}
   		Let $(X,T)$ be a TDS with a metric  $d$. Then
   	$$\overline{\rm mdim}_M(T,X,d)=\limsup_{\epsilon \to 0}\frac{1}{\log \frac{1}{\epsilon}}\sup_{\mu \in M(X,T)}R_{\mu, L^{\infty}}(\epsilon).$$  	
   	Additionally, if $d$ has tame growth of  covering numbers, then for any $1\leq p<\infty$,
   	$$\overline{\rm mdim}_M(T,X,d)=\limsup_{\epsilon \to 0}\frac{1}{\log \frac{1}{\epsilon}}\sup_{\mu \in M(X,T)}R_{\mu, L^p}(\epsilon),$$
   	where $\overline{\rm mdim}_M(T,X,d)$ denotes the upper metric mean dimension of $X$, $R_{\mu, L^p}(\epsilon),  R_{\mu, L^{\infty}}(\epsilon)$	   denote the $L^p, L^{\infty}$ rate-distortion function, respectively.
   \end{thma}

   By replacing rate-distortion functions, the authors  \cite{vv17,twl20,gs20,shi}   verified that   Lindenstrauss-Tsukamoto's variational principles  still hold for  other measure-theoretic  $\epsilon$-entropies.  Inspired by the work of \cite{lw00, lt18, CL21, GZ22},  the aim of present  note is to  introduce the notion of   Feldman-Katok metric mean dimensions and establish  variational principles for Feldman-Katok metric mean dimensions.  The first  question that we  encounter  is whether the metric mean dimensions defined by different metrics have the same metric mean dimension compared with  Feldman-Katok metric mean dimensions.

  A compact metric space  $(X,d) $  is said to  have \emph{weak tame growth of covering numbers} if 
  $$ \lim\limits_{\epsilon \to 0}\epsilon\log \#(X,d,\epsilon)=0.$$
  Obviously, this condition is weaker than tame growth of covering numbers.   The following theorem shows that different metric mean dimensions have the same metric mean dimension.
 \begin{thm}\label{thm 1.1}
 	Let $(X,T)$ be a  TDS  with a metric $d$ admitting weak
 	tame growth of covering numbers. Suppose that $g$ is a mistake function. Then 
 	\begin{align*}
 		\overline{ \rm mdim}_{FK}(T,X,d)=\overline{\rm mdim}_{M}(T,X,d)=\overline{\rm mdim}_{M}(g;T,X).
 	\end{align*}
 	Consequently, if  $d$ has
 	tame growth of covering numbers, then
 	\begin{align*}
 		&\overline{\rm mdim}_{FK}(T,X,d)= \overline{\rm mdim}_{M}(T,X,d)\\
 		=&\overline{\rm mdim}_{\hat{M}}(T,X,d)=\overline{\rm mdim}_{M}(g;T,X).		
 	\end{align*}
  where $\overline{\rm mdim}_{FK}(T,X,d), \overline{\rm mdim}_{M}(T,X,d), \overline{\rm mdim}_{\hat{M}}(T,X,d), \overline{\rm mdim}_{M}(g;T,X)$  are  metric mean dimensions defined by  and FK metric, Bowen metric, mean metric and dynamical balls with mistake function $g$.
 \end{thm}
  The following   variational principle for  Feldman-Katok metric mean dimension in terms of  Katok $\epsilon$-entropies allows us to link the ergodic theory and  metric mean dimension theory.

\begin{thm} \label{thm 1.2}
 Let $(X,T)$ be a TDS with a metric $d$ satisfying weak 
 tame growth of covering numbers. Then for every $\delta \in (0,1)$ 
\begin{align*}
\overline{\rm mdim}_{FK}(T,X,d)&=\limsup_{\epsilon \to 0}\dfrac{1}{\log\frac{1}{\epsilon}}\sup_{\mu \in E(X,T)}h_{\mu,FK}(T,\epsilon, \delta)\\
&=\limsup_{\epsilon \to 0}\dfrac{1}{\log\frac{1}{\epsilon}}\sup_{\mu \in M(X,T)}h_{\mu,FK}(T,\epsilon, \delta)\\
&=\limsup_{\epsilon \to 0}\dfrac{1}{\log\frac{1}{\epsilon}}\sup_{\mu \in M(X)}h_{\mu,FK}(T,\epsilon, \delta),		
\end{align*}
where $h_{\mu, FK}(T,\epsilon, \delta)$  denotes  FK  Katok's $\epsilon$-entropies  of $\mu$.
\end{thm}

 The last variational principle suggests that  metric mean dimension can also be determined by FK   local $\epsilon$-entropy function.      
\begin{thm} \label{thm 1.3}
	Let $(X,T)$ be a TDS with a metric $d$. Then
	\begin{equation*}
		\begin{split}
			\overline{\rm mdim}_{FK}(T,X,d)=&\limsup_{\epsilon \to 0}\dfrac{1}{\log\frac{1}{\epsilon}}\sup_{x \in X}h_{FK}(x,\epsilon),
		\end{split}
	\end{equation*}
	where $h_{FK}(x,\epsilon)$ denotes the FK local $\epsilon$-entropy function of  $x$.
\end{thm}
   We remark that Theorem \ref{thm 1.1}, Theorem \ref{thm 1.2} and Theorem \ref{thm 1.3} also hold for  $	\underline{\rm mdim}_{FK}(T,X,d)$ by changing $\limsup_{\epsilon \to 0}$ into $\liminf_{\epsilon \to 0}$.

	The rest of this paper is organized as follows. In section 2, we  introduce the notions of FK metric mean dimension,  FK Katok  $\epsilon$-entropy and  FK local $\epsilon$-entropy function. In section 3,  we  give the proofs of Theorems \ref{thm 1.1}, \ref{thm 1.2} and \ref{thm 1.3}.

\section{Preliminary}
In subsection 2.1,  analogous to the metric mean dimension defined by  Bowen metric \cite{lw00} we  introduce the notions of Feldman-Katok   metric mean dimensions.   In subsection 2.2, we   introduce the   notions of  FK  Katok's $\epsilon$-entropies  for Borel probability measure  and FK local $\epsilon$-entropy function on $X$ to pursue the variational principle for FK metric mean dimension.

\subsection{Feldman-Katok   metric mean dimensions}

 Fix $x,y$ $\in X$, $n\in\mathbb{N}$,  and $\delta >0$, we define
an $(n,\delta)$-match of $x $ and $y$ to be an \emph{order preserving} (i.e. $\pi(i) < \pi(j)$ whenever $i < j$) bijection $\pi : D(\pi)\rightarrow R(\pi)$ so that $D(\pi), R(\pi) \subset \{0,1,\cdots,n-1 \}$ and  $d(T^{i}x,T^{\pi(i)}y)<\delta$  for every  $i\in D(\pi)$.
Set
\begin{center}
	$\bar{f}_{n,\delta}(x,y)$=$1-\dfrac{1}{n} \max \{ \lvert \pi \rvert\colon \pi  ~\text{is an}~  (n,\delta)\text{-match of}  ~x \text{ and } y \}$,
\end{center}
where   $\lvert \pi \rvert$ denotes the cardinality of  the set $D(\pi)$.

The \emph{Feldman-Katok   metric} (or FK metric for short) on $X$ is given by 
$$	d_{FK_{n}}(x,y)=\inf \{\delta>0\colon\bar{f}_{n,\delta}(x,y)<\delta \}.$$

 Let $Z$ be a non-empty  subset of $X$.  Given $n\in \mathbb{N}$ and $\epsilon>0$, 
 a set $ E\subset X $ is said to be a $FK$-$(n,\epsilon)$ \textit{spanning set} of $ Z $ if for any $x\in Z,$ there exists  
 $y\in$ $ E $\, such that $d_{FK_{n}}(x,y)<\epsilon $.
 Denote by  $r_{FK}(T,Z,n,d,\epsilon)$ the smallest cardinality of $FK$-$(n,\epsilon)$ spanning sets of $Z$. 
 Put $$r_{FK}(T,Z,d,\epsilon)=\limsup_{n \to \infty}\frac{1}{n}\log r_{FK}(T,Z,n,d,\epsilon).$$

 We define  \emph{Feldman-Katok   upper  and lower metric mean dimensions  of   $X$ }  as 
	\begin{align*}
	\overline{\rm mdim}_{FK}(T,X,d)&=\limsup_{\epsilon \to 0}\dfrac{r_{FK}(T,X,d,\epsilon)}{\log\frac{1}{\epsilon}},\\
	\underline{\rm mdim}_{FK}(T,X,d)&=\liminf_{\epsilon \to 0}\dfrac{r_{FK}(T,X,d,\epsilon)}{\log\frac{1}{\epsilon}}.
	\end{align*}
	
	In \cite{CL21},  Cai and Li defined  the  \emph{Feldman-Katok topological entropy of  $X$}  as $h_{FK}(T,X)=\lim\limits_{\epsilon \to 0}r_{FK}(T,Z,d,\epsilon).$   Then  for sufficiently $\epsilon >0$,  one may think of  $$r_{FK}(T,Z,d,\epsilon)\approx 	\overline{\rm mdim}_{FK}(T,X,d)\cdot \logf.$$  Hence,  FK metric mean dimensions   can be  interpreted  as  how fast the term $r_{FK}(T,X,d,\epsilon)$ approximates the infinite  Feldman-Katok topological entropy as $\epsilon \to 0$. 
	
	The authors \cite{w82,ps07,GJ16,CL21} showed that  different  dynamical balls defined by Bowen metrics, mean  metrics, FK metrics,  Bowen balls with mistake function  leads to the same topological entropy. We briefly recall  their definitions and then define metric mean dimension  for these metrics.
	
	\begin{itemize}
		\item Mean metric: the $n$-th mean  metric on $X$ is given by
		$$\bar{d}_n(x,y)=\frac{1}{n}\sum_{j=0}^{n-1}d(T^jx,T^jy).$$
		The  upper metric mean dimension of $X$ defined by mean  metric is given by 
		$$\overline{\rm mdim}_{\hat{M}}(T,X,d)=\limsup_{\epsilon \to 0}\limsup_{n \to \infty}\frac{\log  \hat{r}_{n}(T,X,d,\epsilon)}{n\logf},$$ 
		where   $ \hat{r}_{n}(T,X,d,\epsilon)$ denotes	the smallest cardinality of $(n,\epsilon)$ spanning sets of  $ X $   in mean metric.
		\item  Bowen metric:  the $n$-th mean  metric on $X$ is given by $$d_{n}(x,y)=\max_{0\leq j\leq n-1}d(T^jx,T^jy).$$ 
		The  upper metric mean dimension of $X$ defined by Bowen metric is given by 
		$$\overline{\rm mdim}_{{M}}(T,X,d)=\limsup_{\epsilon \to 0}\limsup_{n \to \infty}\frac{\log  {r}_{n}(T,X,d,\epsilon)}{n\logf},$$ 
		where   $ {r}_{n}(T,X,d,\epsilon)$ denotes	the smallest cardinality of $(n,\epsilon)$ spanning sets of  $ X $   in Bowen metric.
		\item Mistake function: A  non-decreasing unbounded map $g\colon\ \mathbb{N} \rightarrow \mathbb{N}$ is called a \emph{mistake function} if $g(n)<n$ and 
		$$\lim\limits_{n \to \infty }\dfrac{g(n)}{n}=0.$$
		The \emph{mistake Bowen ball } $ B_{n}(g;x,\epsilon) $ centered at $x$ with radius $\epsilon$ and length $n$ w.r.t $g$ is given by
		\begin{align*}
			B_{n}(g;x,\epsilon)= \{y\in X:  \max_{j\in \Lambda}d(T^jx,T^jy)<\epsilon ~ \text{ for some}~\Lambda  \in I(g;n)\},
		\end{align*}
		where  $I(g;n)$ is the set of  all  subsets $\Lambda$ of $\lambda_n$ satisfying
		$|\Lambda|\geq n- g(n)$ and $\Lambda_{n}=\{0,1,\cdots,n-1\}$.  
	   Then    $g(n)$  is the  number 
		how many mistakes  that we are allowed to shadow an orbit of length $n$.  
		
		A set $E\subset X$ is a $ (g;n,\epsilon)$ \emph{spanning set} of $X$ if for any $ x\in X$, there exists $y\in E$ and $\Lambda \in I(g;n)$ such that $\max_{j\in \Lambda}d(T^jx,T^jy)<\epsilon$.  The smallest cardinality of  $ (g;n,\epsilon)$ spanning sets of $X$ is denoted by $ r_{n}(g;T,X,\epsilon)$. Put 
		$$	r(g;T,X,\epsilon)=\limsup\limits_{n \to \infty}\frac{\log r_{n}(g;T,X,\epsilon)}{n}.$$ 
		We define \textit{the upper  metric mean dimension  of $X$ with mistake function $g$} as
		\begin{align*}
			\overline{\rm mdim}_{M}(g;T,X)=\limsup\limits_{\epsilon \to 0}\frac{ r(g;T,X,\epsilon)}{\log \frac{1}{\epsilon}}.
		\end{align*}		
	\end{itemize}

\subsection{ FK Katok $\epsilon$-entropy and  local $\epsilon$-entropy function}
Let $\mu \in  M(X)$, $\epsilon>0$, $n \in\N$ and $\delta \in (0,1)$. 
Put
\begin{align*}
&R_{\mu, FK}(T,n,\delta, \epsilon)\\
=&\min\{\#E: E\subset X  ~\text{and} ~\mu (\cup_{x\in E}B_{FK_n}(x,\epsilon))> 1-\delta \},
\end{align*}
where  $B_{FK_n}(x,\epsilon)$ denotes the ball  with center $x$ and radius $\epsilon$ defined by FK metric $d_{FK_n}$.

Following the idea of \cite{bk83}, we define {FK  Katok's $\epsilon$-entropies  of $\mu$}   as
\begin{align*}
	h_{\mu,FK}(T,\epsilon, \delta)=\limsup_{n\to \infty} \frac{1}{n} \log R_{\mu, FK}(T,n,\delta. \epsilon).
\end{align*}

The entropy function   $h(x)$ (in terms of Bowen metric) was introduced by Ye and Zhang  \cite{YZ07}  to study  uniform entropy points. Besides, they showed that topological entropy of $X$ is equal to the supremum of   ${h(x)}$  over all points of $X$.   Next, we  introduce the notion of local $\epsilon$-entropy function in terms of FK metric to establish a variational principle for FK metric mean dimensions.

Given $\epsilon >0,x\in X$, we define the \emph{ FK local $\epsilon$-entropy function of $x$} as 

\begin{center}
	$ h_{FK}(x,\epsilon) = \inf\limits\{r_{FK}(T,K,d,\epsilon):  K ~\text{is a closed neighborhood of} ~x\}$.
\end{center}

\section{Proofs of main results}
In this section, we prove Theorems \ref{thm 1.1}, \ref{thm 1.2} and \ref{thm 1.3}.

We first give the proof of Theorem \ref{thm 1.1}.

Let  $\mathcal{U}$  be   a finite open cover  of   $X$.   By $\diam(\mathcal{U})=\max_{U \in \mathcal{U}} \diam U$ we denote the \emph{diameter} of $\mathcal{U}$. By $Leb(\mathcal{U})$,  the \emph{Lebesgue number} of $\mathcal{U}$,  we denote   the maximal  positive number $\epsilon>0$  such that every open ball  $B_d(x,\epsilon)=\{y\in X\colon d(x,y)<\epsilon\}$   is contained in some element of $\mathcal{U}$.

\begin{lem}\label{lem 3.1}
	Let $(X,d)$ be a compact metric space. Then for every $\epsilon>0$, there exists a finite open
	cover $\mathcal{U}$ of $X$  with $\lvert \mathcal{U}  \rvert=\#(X,d,\frac{\epsilon}{4})$. such that $\diam(\mathcal{U})\leq\epsilon$ and $Leb(\mathcal{U})\geq \frac{\epsilon}{4}$.
\end{lem}

\begin{proof}
	Let $ Z$ be a subset of $X$ so that $X=\cup_{x\in Z}B_d(x,\frac{\epsilon}{4})$ with smallest cardinality  $\#(X,d,\frac{\epsilon}{4})$. 
	Then $\mathcal{U}=\{B(x,\frac{\epsilon}{2})\colon x\in Z\}$ is the open cover  that we need.
\end{proof}
\begin{proof}[Proof of theorem \ref{thm 1.1}] We divide the proof into two steps.
	
Step 1. we show 
	 $$\overline{\rm mdim}_{FK}(T,X,d)=	\overline{\rm mdim}_{M}(T,X,d).$$

The inequality  $\overline{\rm mdim}_{FK}(T,X,d)\leq \overline{\rm \rm mdim}_{M}(T,X,d)$ holds by using the fact  $d_{FK_{n}}\leq d_{n}$. By Lemma \ref{lem 3.1}, there exists a finite open
cover $\mathcal{U}$ of $X$ with $\lvert \mathcal{U}  \rvert=\#(X,d,\frac{\epsilon}{4})$ such that $\diam(\mathcal{U})\leq\epsilon$ and $Leb(\mathcal{U})\geq \frac{\epsilon}{4}$.
Let $E_1$ be a $
FK$-$(n,\frac{\epsilon}{4})$ spanning set  of $X$ with  the  smallest cardinality $r_{FK}(T,X,n,d,\frac{\epsilon}{4})$. Then 
\begin{align}\label{inequ 3.1}
X=\mathop{\cup}_{x\in E_1}\mathop{\cup}_{k=[(1-\frac{\epsilon}{4})n]}^{n}\mathop{\cup}_{\pi: \lvert \pi \rvert=k, \atop \pi \text{ is  order  preserving}}
\mathop{\cap}_{i\in D(\pi)}T^{-i} B_d(T^{\pi(i)}x,\frac{\epsilon}{4}).
\end{align}

Since each open ball  $B_d(T^{\pi(i)}x,\frac{\epsilon}{4})$ is contained in some element of $\mathcal{U}$, then   $\mathop{\cap}_{i\in D(\pi)}\limits T^{-i} B_d(T^{\pi(i)}x,\frac{\epsilon}{4})$
is contained in some element of $\bigvee\limits_{i\in D(\pi)}T^{-i}\mathcal{U}$.
Note that
$\bigvee\limits_{i=0}^{n-1}T^{-i}\mathcal{U}=(\bigvee\limits_{i\in D(\pi)}T^{-i}\mathcal{U})\vee (\bigvee\limits_{i\notin D(\pi)}T^{-i}\mathcal{U})$
and $ |\bigvee\limits_{i\notin D(\pi)}T^{-i}\mathcal{U}|\leq |\mathcal{U}|^{n-|\pi|}.$ Then 
 $\mathop{\cap}\limits_{i\in D(\pi)}T^{-i}B_d(T^{\pi(i)}x,\frac{\epsilon}{4})$ can  be  at most  covered by $ \rvert \mathcal{U} \rvert ^{n-\mid\pi\mid} $ elements of $\bigvee\limits_{i=0}^{n-1}T^{-i}\mathcal{U}$.
Since the number of order preserving bijection $ \pi $ with $\rvert \pi^{} \rvert$= $ k $ is not more than $(C_{n}^{k})^{2}$, then $ X $ can be covered by 	$\rvert E_1 \rvert \sum\limits_{k=[(1-\frac{\epsilon}{4})n]}^{n}(C_{n}^{k})^{2}\lvert \mathcal{U} \rvert ^{n-k}$
elements of $\bigvee\limits_{i=0}^{n-1}T^{-i}\mathcal{U}$. By $ N(\mathcal{U}) $ we  denote the smallest cardinality  of  subcover of $\mathcal{U}$  covering $X$.  Recall that the topological entropy of $\mathcal{U}$ \cite{w82} is given by  $ h_{top}(T,\mathcal{U})=\lim\limits_{n \to \infty }\frac{1}{n}\log N(\bigvee\limits_{i=0}^{n-1}T^{-i}\mathcal{U})$.  By (\ref{inequ 3.1}), this  yields that 
\begin{equation}\label{equ 3.2}
\begin{aligned}
	N(\bigvee\limits_{i=0}^{n-1}T^{-i}\mathcal{U}) & \leq \rvert E_1 \rvert \sum\limits_{k=[(1-\frac{\epsilon}{4})n]}^{n}(C_{n}^{k})^{2}\lvert \mathcal{U} \rvert ^{n-k}                                       \\
	& \leq \rvert E_1 \rvert \sum\limits_{k=[(1-\frac{\epsilon}{4})n]}^{n}4^{n}\lvert \mathcal{U} \rvert ^{n-k}                                       \\
	& \leq r_{FK}(T,X,n,d,\frac{\epsilon}{4})\cdot 4^{n}\cdot\lvert \mathcal{U} \rvert ^{\frac{n\epsilon}{4} +1}\cdot(\frac{n\epsilon}{4} +1).
\end{aligned}
\end{equation}

It follows that  $h_{top}(T,\mathcal{U})\leq r_{FK}(T,X,d,\frac{\epsilon}{4})+\frac{\epsilon}{4}\log|\mathcal{U}|+\log4.$ Since $\diam(\mathcal{U})\leq \epsilon$, one has 
$$	r(T,X,d,2\epsilon):= \limsup_{n \to \infty}\frac{\log  {r}_{n}(T,X,d,2\epsilon)}{n}\leq h_{top}(T,\mathcal{U})$$ and hence
\begin{align}
	r(T,X,d,2\epsilon)\leq r_{FK}(T,X,d,\frac{\epsilon}{4})+\frac{\epsilon}{4}\log \#(X,d,\frac{\epsilon}{4})+\log4.
\end{align}
Since  $d$ has weak tame  growth, one has $\overline{\rm  mdim}_{M}(T,X,d)\leq \overline{\rm mdim}_{FK}(T,X,d).$

Step 2. We  continue to show
\begin{equation*}
\begin{split}
\overline{\rm mdim}_{M}(T,X,d)=\overline{\rm mdim}_{M}(g;T,X)
\end{split}
\end{equation*}

The inequality $\overline{\rm mdim}_{M}(g;T,X)\leq\overline{\rm mdim}_{M}(T,X,d)$ follows by 
the fact that $B_{n}(x,\epsilon)\subset B_{n}(g;x,\epsilon)$. Fix   $\epsilon>0$. Let $E_{2}$ be a
$(g;n,\frac{\epsilon}{4})$ spanning set of $X$ with $ \lvert E_{2}  \rvert= r_{n}(g;T,X,\frac{\epsilon}{4})$. So we have
\begin{center}
	$X=\bigcup\limits_{x\in E_{2}}\bigcup\limits_{k=[n-g(n)]}^{n}\bigcup\limits_{|\Lambda|=k}
	\bigcap\limits_{i\in \Lambda}T^{-i} B{(T^{i}x,\frac{\epsilon}{4})}.$
\end{center}

Similar to Step 1,  one can get that 
\begin{align*}
N(\bigvee\limits_{i=0}^{n-1}T^{-i}\mathcal{U}) & \leq \rvert E_{2} \rvert \sum\limits_{k=[n-g(n)]}^{n}C_{n}^{k}\lvert \mathcal{U} \rvert ^{n-k}\\
& \leq r_{n}(g;T,X,\frac{\epsilon}{4})\cdot (g(n)+1)\cdot\lvert \mathcal{U} \rvert ^{g(n)+1}\cdot 2^{n}.
\end{align*}
Therefore,
\begin{equation*}
\begin{aligned}
\frac{1}{n}\log N(\bigvee\limits_{i=0}^{n-1}T^{-i}\mathcal{U}) & \leq{\frac{\log r_{n}(g;T,X,\frac{\epsilon}{4}) }{n}}+{\frac{\log (g(n)+1) }{n}}                      \\
& +\log 2+\frac{ (g(n)+1)\log r_{1}(T,X,d,\frac{\epsilon}{4}) }{n}.
\end{aligned}
\end{equation*}
So $r(T,X,d, 2\epsilon)\leq h_{top}(T,\mathcal{U})\leq r(g;T,X,\frac{\epsilon}{4})+\log 2 .$  This shows that
$\overline{\rm mdim}_{M}(T,X,d)\leq \overline{\rm mdim}_{M}(g;T,X).$

If  $d$ has
tame growth of covering numbers, we have $\overline{\rm mdim}_{M}(T,X,d)=\overline{\rm mdim}_{\hat{M}}(T,X,d)$ \cite{lt18}. Together with  Steps 1 and 2,  this finishes the proof.
\end{proof}

In  fact,  by  the proof of  Step 2, we have  $\overline{\rm mdim}_{M}(T,X,d)= \overline{\rm mdim}_{M}(g;T,X)$ without the assumption of weak tame growth for $d$. Next, we proceed to   give the proof  of Theorem \ref{thm 1.2}.

\begin{proof}[Proof of Theorem \ref{thm 1.2}]
	  Fix $\epsilon>0$, $\delta \in (0,1)$ and let $\mu \in M(X)$.
	Define
	$$h_{\mu}(T,\epsilon, \delta)=\limsup_{n\to \infty} \frac{1}{n} \log R_\mu(T,n,\delta. \epsilon),$$
	where
	$R_\mu(T,n,\delta, \epsilon):=\min\{\#E: E\subset X  ~\text{and} ~\mu (\cup_{x\in E}B_{d_n}(x,\epsilon))> 1-\delta \}$.

	Let $ Z=\{z_1,...,z_l\}$ be a subset of $X$ so that $X=\cup_{1\leq j\leq l}B_d(x_j,\frac{\epsilon}{2})$ with smallest cardinality  $l=\#(X,d,\frac{\epsilon}{2})$.  Let  $E$ be a subset of $X$  with the smallest cardinality  $R_{\mu, FK}(T,n,\delta. \frac{\epsilon}{4})$ so that  $\mu (\cup_{x\in E}B_{FK_n}(x,\frac{\epsilon}{4}))> 1-\delta$.  Note that for each $x\in E$, 
	\begin{equation}\label{equ 3.4}
	B_{FK_n}(x,\frac{\epsilon}{4})=\mathop{\cup}_{k=[(1-\frac{\epsilon}{4})n]}^{n}\mathop{\cup}_{\pi\colon
		\lvert \pi \rvert=k, \atop \pi \text{ is  order  preserving}}
	\mathop{\cap}_{i\in D(\pi)}T^{-i} B_d(T^{\pi(i)}x,\frac{\epsilon}{4}).
	\end{equation}
	Choose $x_{y}\in \cap_{i\in D(\pi)}T^{-i} B_d(T^{\pi(i)}x,\frac{\epsilon}{4})$ with $|\pi|=k$. Then
	\begin{align}\label{equ 3.5}
	A:=\mathop{\cap}_{i\in D(\pi)}T^{-i} B_d(T^{\pi(i)}x,\frac{\epsilon}{4})\subset \mathop{\cap}_{i\in D(\pi)}T^{-i} B_d(T^{i}x_y,\frac{\epsilon}{2}).
	\end{align}
	Let  $\{0,...,n-1\}\backslash D(\pi)=\{j_1,j_2,...,j_a\}$ with $a=n-k$. Then 
	\begin{equation}\label{equ 3.6}
	\begin{aligned}
	A&\subset\mathop{\cup}_{1\leq m_1,...,m_a\leq l }\left(\cap_{i=1}^a T^{-j_i} B_d(z_{m_i},\frac{\epsilon}{2}) \cap A\right)\\
	&\subset \mathop{\cup}_{1\leq m_1,...,m_a\leq l }\cap_{i=0}^{n-1}T^{-i} B_{d}(x_{y,z_{m_1},...,z_{m_a}},\epsilon)
	\end{aligned}
	\end{equation}
	for some $x_{y,z_{m_1},...,z_{m_a}}\in \cap_{i=1}^a T^{-j_i} B_d(z_{m_i},\frac{\epsilon}{2}) \cap A\not=\emptyset$. Therefore, by (\ref{equ 3.4}), (\ref{equ 3.5}) and (\ref{equ 3.6}) one has 
	$$R_\mu(T,n,\delta. \epsilon)\leq  R_{\mu, FK}(T,n,\delta. \frac{\epsilon}{4}) \sum\limits_{k=[(1-\frac{\epsilon}{4})n]}^{n}(C_{n}^{k})^{2}(\#(X,d,\frac{\epsilon}{2}))^{n-k}.$$
	Similar to (\ref{equ 3.2}), we have 
	\begin{align}\label{equ 3.7}
	h_{\mu}(T,\epsilon, \delta)\leq h_{\mu,FK}(T,\frac{\epsilon}{4}, \delta)+ \frac{\epsilon}{4}\log \#(X,d,\frac{\epsilon}{2})+\log4.
	\end{align}

Hence,	
	\begin{align*}
		\overline{ \rm mdim}_{FK}(T,X,d)&=\overline{\rm mdim}_{M}(T,X,d)~ \text{by Theorem \ref{thm 1.1}}\\
		&=\limsup_{\epsilon \to 0}\dfrac{1}{\log\frac{1}{\epsilon}}\sup_{\mu \in E(X,T)}h_{\mu}(T,\epsilon, \delta)~\text{by \cite[Theorem 4.2]{shi}}\\
		&\leq\limsup_{\epsilon \to 0}\dfrac{1}{\log\frac{1}{\epsilon}}\sup_{\mu \in E(X,T)}h_{\mu,FK}(T,\epsilon, \delta)~\text{by (\ref{equ 3.7})}.
	\end{align*}
One the other hand,  $h_{\mu,FK}(T,\epsilon, \delta)\leq r_{FK}(T,X,d,\epsilon)$ holds for every $\mu \in M(X)$.  We complete  the proof.
\end{proof}
Finally,  we give the proof  of Theorem \ref{thm 1.3}.

\begin{prop}\label{prop 3.2}
	Let $ ( X,T) $ be a TDS with a metric $d$. Suppose that  $Z_{1},Z_{2},\cdots,  Z_{m}$ are closed subsets of $X$. Then for  every  $\epsilon >0$, 
	$$r_{FK}(T,\cup_{j=1}^{m} Z_{j},d,\epsilon)=\max_{1\leq
		j\leq m} r_{FK}(T,Z_{j},d,\epsilon).$$
\end{prop}
\begin{proof}
	Fix  $\epsilon>0$. It suffices to show  $$r_{FK}(T,\cup_{j=1}^{m} Z_{j},d,\epsilon)\leq\max_{1\leq
		j\leq m} r_{FK}(T,Z_{j},d,\epsilon).$$ 
	For  every $n\in \mathbb{N}$, one  can choose $1\leq j_{(n,\epsilon)}\leq m$ such that
	 $$\max_{1\leq j\leq m} r_{FK}(T,Z_{j},n,d,\epsilon)=r_{FK}(T,Z_{j_{(n,\epsilon)}},n,d,\epsilon).$$ 
	 and  hence $	r_{FK}(T,\cup_{j=1}^{m} Z_{j},n,d,\epsilon)\leq m \cdot  r_{FK}(T,Z_{j_{(n,\epsilon)}},n,d,\epsilon).$ This implies that  $$\log r_{FK}(T,\cup_{j=1}^{m} Z_{j},n,d,\epsilon)\leq \log m+\log r_{FK}(T,Z_{j_{(n,\epsilon)}},n,d,\epsilon).$$
	 By Pigeon principle and the  definition of  $ r_{FK}(T,\cup_{j=1}^{m} Z_{j},d,\epsilon)$,  we can
	choose a subsequence  $n_{k}\to \infty $ such that
	$$ \frac{1}{n_{k}}\log sp_{FK}(T,\cup_{j=1}^{m} Z_{j},n_{k},d,\epsilon)\rightarrow  r_{FK}(T,\cup_{j=1}^{m} Z_{j},d,\epsilon)$$
	and  $Z_{j_{(n_{k},\epsilon)}}=Z_{j_\epsilon}$  for all $k$, where $1\leq j_{\epsilon} \leq m$ is a constant  independent of the choice of  $n_k$.
	It follows that 
	$$r_{FK}(T,\cup_{j=1}^{m} Z_{j},d,\epsilon)\leq r_{FK}(T,Z_{j_\epsilon},d,\epsilon)\leq\max_{1\leq
		i\leq m} r_{FK}(T,Z_{i},d,\epsilon).$$
\end{proof}

\begin{proof}[Proof of Theorem \ref{thm 1.3}]
	
	Fix $\epsilon >0$.  It is clear that  $\sup_{x \in X}h_{FK}(x,\epsilon)\leq  r_{FK}(T,X,d,\epsilon)$. 
	
	Let $\{B_{1}^{1},\cdots, B_{m_{1}}^{1}\}$  be a finite  closed balls family of  $X$  with radius at most $1$.  By Proposition \ref{prop 3.2}, there exists $1\leq j_{1}\leq m_1$ such that
$$r_{FK}(T,X,d,\epsilon)= r_{FK}(T,B^{1}_{j_{1}},d,\epsilon).$$
	
	For  the closed ball $  B^{1}_{j_{1}} $, let $\{B_{1}^{2},\cdots, B_{m_{2}}^{2}\}$ be   a finite  closed balls family of $  B^{1}_{j_{1}} $ with radius  at most $\frac{1}{2}$ covering $ B^{1}_{j_{1}}$ and $B_{i}^{2} \subset B^{1}_{j_{1}} $ for every $1\leq i \leq  m_2$. 
	  Then by Proposition \ref{prop 3.2} again  there exists $1\leq j_{2}\leq m_2$ such that
	$$r_{FK}(T, B^{1}_{j_{1}},d,\epsilon)= r_{FK}(T, B^{2}_{j_{2}},d,\epsilon).$$
     Repeating this procedure, for  every $n\geq 2$, there exists a closed ball $ B^{n}_{j_{n}} \subset B_{j_{n-1}}^{n-1} $ with radius at most $\frac{1}{n}$ such that
	$$r_{FK}(T,X,d,\epsilon)=r_{FK}(T,\cap_{i=1}^{n} B^{i}_{j_{i}},d,\epsilon).$$
	
	Let  $\{x_{0}\}=\cap_{n\geq 1}B^{n}_{j_{n}}$.  For any closed neighborhood $ K $ of
	$ x_{0} $, we can choose $ n_0 $ so that
	$\cap_{i=1}^{n_0 } B^{i }_{j_i}\subset K.$ Therefore,
	\begin{align*}
		r_{FK}(T,X,d,\epsilon)=r_{FK}(T,\cap_{i=1}^{n_0} B^{i}_{j_{i}},d,\epsilon)\leq r_{FK}(T,K,d,\epsilon),
	\end{align*}	
  which implies that   $r_{FK}(T,X,d,\epsilon)\leq  h_{FK}(x_{0},\epsilon) \leq \sup_{x \in X}h_{FK}(x,\epsilon)$.
\end{proof}

\section*{Disclosure statement}

No potential conflict of interest was reported by the authors.
\section*{Acknowledgement} 

\noindent  The second author was supported by the National Natural Science Foundation of China (Nos.12071222 and 11971236).  The third author  was supported by Postgraduate Research $\&$ Practice Innovation Program of Jiangsu Province (No. KYCX23$\_$1665).  The work was also funded by the Priority Academic Program Development of Jiangsu Higher Education Institutions.  We would like to express our gratitude to Tianyuan Mathematical Center in Southwest China(No.11826102), Sichuan University and Southwest Jiaotong University for their support and hospitality.  


\begin{thebibliography}{HD82}
	
	\normalsize
	\baselineskip=15pt
	
	\bibitem[BK83]{bk83}M. Brin and A. Katok,  On local entropy,  Geometric dynamics (Rio de Janeiro),  \emph {Lecture Notes in Mathematics},  Springer, Berlin 
	\textbf{1007}(1983), 30-38.
	
	\bibitem[CL21]{CL21} F. Cai and J. Li, On Feldman-Katok metric and entropy formulae. arXiv: 2104.12104.
	
	
	\bibitem[Fel76]{F76} J. Feldman,  New K-automorphisms and a problem of Kakutani,  \emph{Israel J. Math.} \textbf{24}(1976), 16-38.
	
	\bibitem[GZ22]{GZ22}K. Gao and R. Zhang,  On variational principles of metric mean dimension on subset in Feldman-Katok metric, to appear in \emph{Acta Math. Sin. (Engl. Ser.)}, 2023.
	
	
	\bibitem[GJ16]{GJ16} M. Gr$\rm{\ddot{o}}$ger and T. J$\rm{\ddot{a}}$ger,  Some remarks on modified power entropy, \emph{Contemp. Math.} \textbf {669}(2016), 105-122.
	
	\bibitem[Gro99]{gromov} M. Gromov, Topological invariants of dynamical systems and spaces of holomorphic maps: I, \emph{Math. Phys, Anal. Geom.} \textbf{4}(1999), 323-415.
	
	\bibitem[GS21]{gs20} Y. Gutman and A. \'Spiewak,  Around the variational principle for metric mean dimension,  \emph {Studia Math.} \textbf{261}(2021), 345-360.
		
	
	\bibitem[Kak43]{K43}S. Kakutani,  Induced measure preserving transformations, \emph{Proc. Imp. Acad. Tokyo.} \textbf{19}(1943), 635-641.
	
	\bibitem[KL17]{KL17}K. Dominik and L.  Martha, Feldman-Katok pseudometric and the GIKN construction of non-hyperbolic ergodic measures, arXiv: 1702.01962.
			
	\bibitem[Kat80]{k80} A. Katok,  Lyapunov exponents, entropy and periodic orbits for diffeomorphisms, \emph{Publ. Math. Inst. Hautes \'Etudes Sci.} \textbf{51}(19jg80), 137-173.
	
	
	\bibitem[LT18]{lt18} E. Lindenstrauss  and M. Tsukamoto, From rate distortion theory to metric mean dimension: variational principle, \emph{IEEE Trans. Inform. Theory} \textbf{64}(2018), 3590-3609.
	
	
	\bibitem[LW00]{lw00} E. Lindenstrauss and B. Weiss,  Mean topological dimension, \emph{Isr. J. Math.} \textbf{115}(2000), 1-24. 
	
	\bibitem[NH22]{NH22}X. Nie and Y. Huang,  Restricted sensitivity, return time and entropy in Feldman-Katok and mean metrics, \emph{Dyn. Syst.} \textbf{37}(2022),  357-381. 
	
	\bibitem[Orn74]{O74}D. Ornstein,  Ergodic theory, randomness, and dynamical systems, 1974.
	
	
	\bibitem[PS07]{ps07} C. Pfister and W. Sullivan, On the topological entropy of saturated sets, \emph{Ergodic Theory Dynam. Syst.} \textbf{27}(2007), 929-956.
		
	
	\bibitem[Shi22]{shi} R. Shi, On variational principles for metric mean dimension, \emph{IEEE Trans. Inform. Theory} \textbf{68}(2022), 4282-4288.
	
	\bibitem[TWL20]{twl20} D. Tang, H. Wu and Z. Li,  Weighted upper metric mean dimension for amenable group actions, \emph{Dyn. Syst.}  \textbf{35}(2020), 382-397.
	
	
	\bibitem[Tsu20]{t20} M. Tsukamoto, Double variational principle for mean dimension with potential, \emph{Adv. Math.} \textbf{361}(2020), 106935, 53 pp. 
	
	
	\bibitem[VV17]{vv17}  A. Velozo and R. Velozo,   Rate distortion theory, metric mean dimension and measure theoretic entropy, arXiv:1707.05762.
		
	\bibitem[Wal82]{w82} P. Walters,  An introduction to ergodic theory, Springer, 1982.
	
	\bibitem[YZ07]{YZ07}X. Ye and G. Zhang,  Entropy points and applications,  \emph{Trans. Amer. Math. Soc.} \textbf{359}(2007), 6167-6186.
		

	
\end{thebibliography}

\end{document}